\documentclass[12pt]{article}
\makeatletter
\newcommand{\rmnum}[1]{\romannumeral #1}
\newcommand{\Rmnum}[1]{\expandafter\@slowromancap\romannumeral #1@}
\makeatother
\usepackage{}
\usepackage[numbers,sort&compress]{natbib}

\usepackage{color}
\usepackage{threeparttable}

\usepackage[english]{babel}
\usepackage{diagbox}
\usepackage{t1enc}
\usepackage{amsthm,amsmath,amssymb}
\usepackage{mathrsfs}
\usepackage[mathscr]{eucal}
\usepackage[latin1]{inputenc}
\usepackage[english]{babel}
\usepackage{latexsym}
\usepackage{graphicx}
\usepackage[noend]{algpseudocode}
\usepackage{algorithmicx,algorithm}
\usepackage{booktabs}
\usepackage{multirow}
\newtheorem{theorem}{Theorem}
\newtheorem{corollary}{Corollary}
\newtheorem{proposition}{Proposition}
\newtheorem{lemma}{Lemma}
\newtheorem{remark}{Remark}

\newtheorem{definition}{Definition}

\def \ord{\textup{ord}}
\def \col{\textup{col}}

\def \sfp{\textup{sfp}}
\def \diag{\textup{diag}}

\def \T{\textup{T}}
\def \nullity{\textup{nullity}}
\def\rank{\textup{rank}}
\def\nullity{\textup{nullity}}
\def \rrg{\textup{RO}_n(\mathbb{Q})}
\begin{document}
	\title{Primary decomposition theorem and generalized spectral characterization of graphs}
\author{\small	Songlin Guo$^{{\rm a}}$ \quad Wei Wang$^{{\rm a}}$\thanks{Corresponding author. Email address: wangwei.math@gmail.com}\quad Wei Wang$^{{\rm b}}$
	\\
{\footnotesize$^{\rm a}$School of Mathematics, Physics and Finance, Anhui Polytechnic University, Wuhu 241000, P. R. China}\\
{\footnotesize$^{\rm b}$School of Mathematics and Statistics, Xi'an Jiaotong University, Xi'an 710049, P. R. China}
}
\date{}

\maketitle
\begin{abstract}
	 Suppose $G$ is a controllable graph of order $n$ with adjacency matrix $A$. Let  $W=[e,Ae,\ldots,A^{n-1}e]$ ($e$ is the all-one vector) and $\Delta=\prod_{i>j}(\alpha_i-\alpha_j)^2$ ($\alpha_i$'s are eigenvalues of $A$) be the walk matrix and the discriminant of $G$, respectively. Wang and Yu \cite{wangyu2016} showed that if 
	 $$\theta(G):=\gcd\{2^{-\lfloor\frac{n}{2}\rfloor}\det W,\Delta\} $$
 is odd and squarefree, then $G$ is determined by its generalized spectrum (DGS). Using the primary decomposition theorem, we obtain a new criterion for a graph $G$ to be DGS without the squarefreeness assumption on $\theta(G)$. Examples are further given to illustrate the effectiveness of the proposed criterion, compared with the two existing methods to deal with the difficulty of non-squarefreeness.
	
\end{abstract}

	\noindent\textbf{Keywords:} Generalized spectrum; Controllable graph; Totally isotropic; Rational orthogonal matrix. \\
	
	\noindent\textbf{Mathematics Subject Classification:} 05C50

\section{Introduction}
Let $G$ be a simple graph with vertex set $\{v_1,\ldots,v_n\}$. The \emph{adjacency matrix} of $G$ is the symmetric (0,1)-matrix $A=A(G)=(a_{i,j})$, where $a_{i,j}=1$ precisely when $v_i$ and $v_j$ are adjacent. We often identify a graph $G$ with its adjacency matrix $A$.  The \emph{spectrum} of a graph $G$, denoted by $\sigma(G)$,  is the multiset of eigenvalues of $A$.  The \emph{generalized spectrum} of $G$ is the ordered pair $(\sigma(G),\sigma(\bar{G}))$,
where $\bar{G}$ is the complement of $G$. Two graphs are \emph{generalized cospectral} if they share the same generalized spectrum; a graph $G$ is \emph{determined by its generalized spectrum} (DGS) if any graph generalized cospectral with $G$ is isomorphic to $G$.

The problem of determining whether a given graph is DGS has received considerable attention in recent years \cite{abiad,qiu,wang2013EJC,wang2017JCTB,wang2023Eujc}. Some sufficient conditions for DGS-property  have been obtained based on the walk matrix and the discriminant of a graph. For a graph $G$ with adjacency matrix $A$, the \emph{walk matrix} of $G$ is the square integral matrix $$W=W(G)=[e,Ae,\ldots,A^{n-1}e],$$ where $e$ is the all-one vector. The \emph{discriminant} \cite{wangyu2016} of $G$, denoted by $\Delta(G)$ or $\Delta(A)$,  is defined to be the discriminant of its characteristic polynomial $\det(xI-A)$. That is, $$\Delta(G)=\prod_{i>j}(\alpha_i-\alpha_j)^2,$$ where $\alpha_1,\ldots,\alpha_n$ are eigenvalues of $G$. Note that $\Delta(G)$ is always an even integer \cite{wangyu2016}. For the walk matrix $W$, it is known that $2^{-\lfloor\frac{n}{2}\rfloor}\det W$ is always an integer \cite{wang2013EJC}. For convenience, we introduce an integral invariant by combining $\det W(G)$ and $\Delta(G)$:
 \begin{definition}\normalfont{
	For an $n$-vertex graph $G$, we define 
	\begin{equation}
	\theta(G)=\gcd\left\{2^{-\lfloor\frac{n}{2}\rfloor }\det W(G),\Delta(G)\right\},
	\end{equation}
	where $\gcd\{\cdot,\cdot\}$ denotes the greatest common divisor of two integers. }
\end{definition}
\begin{remark}
	$\theta(G)$ is invariant under generalized cospectrality. That is, if $H$ is generalized cospectral with $G$, then 	$\theta(G)=\theta(H)$.
\end{remark}

A graph $G$ is \emph{controllable} \cite{godsil2012} if $\det W\neq 0$.  O'Rourke and Touri \cite{rourke} showed that almost all graphs are controllable, confirming a conjecture of Godsil \cite{godsil2012}.  We use $\mathcal{G}_n^c$ to denote the set of all controllable graphs of order $n$ and we only consider such graphs in this paper. We mention that a controllable graph can only have simple eigenvalues and hence has a nonzero discriminant. 
 
 The following arithmetic criterion for DGS-property was reported in \cite{wangyu2016}, improving the previous results obtained in \cite{wang2013EJC} and \cite{wang2017JCTB}.
\begin{theorem}[\cite{wang2013EJC, wang2017JCTB, wangyu2016}]\label{old}
Let $G\in \mathcal{G}_n^c$. If $\theta(G)$ is odd and squarefree then $G$ is DGS.
\end{theorem}
The main aim of this paper is to give a new criterion for DGS-property without the  squarefreeness assumption on $\theta(G)$. We only assume $\theta(G)$ to be odd and allow $\theta(G)$ to contain one or more multiple prime factors. Let $p$ be such a factor, which is of course a multiple factor of $\det W$.  Indeed, some criteria (sufficient conditions) for DGS-property have already been obtained in this regard. One is  the Exclusion Condition \cite{wang2006EUJC} obtained under the assumption that $\rank_p W=n-1$; the other is the Improved Condition \cite{wang2023Eujc} obtained under the assumption that the last invariant factor $d_n(W)$ of $W$ is squarefree.

We emphasize, however, that neither of the above two additional assumptions is needed in this paper. To state our result compactly, we need some notations, the first of which was introduced in \cite{wang2023Eujc}.
\begin{definition}[\cite{wang2023Eujc}]\label{Phip}
	\normalfont{
	Let $p$ be an odd prime and $G$ be a graph with adjacency matrix $A$. We define	
	\begin{equation}\Phi_p(G;x)=\gcd\left\{\chi(A;x),\chi(A+J;x)\right\}\in \mathbb{F}_p[x],
	\end{equation}
	where $J$ is the all-one matrix,  $\chi(A;x)=\det(xI-A)$ is the characteristic polynomial of $A$, and the  greatest common divisor (gcd) is taken over $\mathbb{F}_p$.}
\end{definition}
For an integer $m$ and a prime $p$, we define 
\begin{equation}
\ord_p(m)=\begin{cases}
\max\{k\colon\,k\ge 0 \text{~and~} p^k\mid m\} &\text{if $m\neq 0$,}\\
\infty&\text{if $m=0$.}
\end{cases}
\end{equation} The main result of this paper is the following theorem.
\begin{theorem}\label{main}
 Let $G\in \mathcal{G}_n^c$ such that $\theta(G)$ is odd.
If for each multiple prime factor $p$ of $\theta(G)$ and for each multiple irreducible factor $\phi(x)$ of $\Phi_p(G;x)$,
\begin{equation}\label{keyequ}
\min\{\ord_p(\det \phi(A)),\ord_p(\det \phi(A+J))\}=\deg \phi(x)
\end{equation}
then $G$ is DGS.	
\end{theorem}
\begin{remark}
	In the left hand side of Eq.~\eqref{keyequ}, we should regard $\phi(x)\in\mathbb{Z}[x]$ with each coefficient in $\{0,1,\ldots,p-1\}$. We shall make the same convention when determinants are concerned, e.g., in Eqs.~\eqref{basicine} and \eqref{relex}.
\end{remark}
The following simplified criterion for DGS-property is immediate from Theorem \ref{main}.
\begin{corollary}
 Let $G\in \mathcal{G}_n^c$ such that $\theta(G)$ is odd. If for each multiple prime factor $p$ of $\theta(G)$, $\Phi_p(G;x)$ is squarefree, i.e.,
  \begin{equation}\gcd\left\{\Phi_p(G;x),\Phi'_p(G;x)\right\}=1,\text{~over~}\mathbb{F}_p, 
 \end{equation}
 then $G$ is DGS.	
\end{corollary}
\section{Preliminaries}
\subsection{Inner product space over $\mathbb{F}_p$}
Let $p$ be an odd prime. We use $\mathbb{F}_p^n$ to denote the standard $n$-dimentional linear space $\{(x_1,x_2,\ldots,x_n)^\T\colon\, x_i\in \mathbb{F}_p\}$ equipped with the standard inner product: $\langle u,v \rangle=u^\T v$. If $u$ and $v$ are two vectors in  $\mathbb{F}_p^n$ satisfying $u^\T v=0$, we shall say $u,v$ are \emph{orthogonal} and denote the relation by the ordinary notation $u\perp v$. Similarly, two subspaces $U$ and $V$ are \emph{orthogonal}, denoted by $U\perp V$, if $u\perp v$ for any $u\in U$ and  $v\in V$.  A nonzero vector $u\in \mathbb{F}_p^n$ is \emph{isotropic} if $u\perp u$, i.e., $u^\T u=0$.  For a subspace $V$ of $\mathbb{F}_p^n$, the \emph{orthogonal space} of $V$ is
\begin{equation}
V^\perp=\{u\in \mathbb{F}_p^n\colon\, v^\T u=0\text{~for every $v\in V$}\}.
\end{equation}
A subspace $V$ of $\mathbb{F}_p^n$ is \emph{totally isotropic} \cite{babai2022} if $V\subset V^\perp$, i.e., every pair of vectors  in $V$ are orthogonal.

Let $A$ be an $n\times n$ matrix over $\mathbb{F}_p$. We regard $A$ as a linear transformation  that maps $x$ to $Ax$ for each vector $x\in \mathbb{F}_p^n$.
A subspace $V$ of $\mathbb{F}_p^n$ is $A$-\emph{invariant} if $Ax\in V$ for any $x\in V$. For an $A$-invariant subspace $V$, we use $A|_V$ to denote the restriction of the linear transformation $A$ on $V$. We need the following basic result of Linear Algebra, which is usually referred to as Primary Decomposition Theorem; see e.g. \cite{hoffman}. 
\begin{theorem}[Primary Decomposition Theorem]\label{pft}
Let $\chi(A;x)\in \mathbb{F}_p[x]$ be the characteristic polynomial of $A$ and
$$\chi(A;x)=\phi_1^{r_1}(x)\phi_2^{r_2}(x)\cdots\phi_k^{r_k}(x),$$
be the standard factorization of $\chi(A;x)$. Let $V_i=\ker \phi^{r_i}_i(A), i = 1,\ldots, k$. Then\\	
(\rmnum{1}) $\mathbb{F}_p^n = V_1\oplus V_2\oplus\cdots\oplus V_k$; \\
(\rmnum{2}) each $V_i$ is $A$-invariant;  \\
(\rmnum{3}) the characteristic polynomial of $A|_{V_i}$ is $\phi_i^{r_i}(x)$;\\
(\rmnum{4}) there are polynomials $h_1(x),\ldots,h_k(x)$ such that each $h_i(A)$ is identity on $V_i$ and is zero on all the other $V_i$'s;\\
(\rmnum{5}) for any $A$-invariant subspace $U$, we have $U=(U\cap V_1)\oplus (U\cap V_2)\oplus\cdots\oplus(U\cap V_k)$;\\
\end{theorem}
\begin{lemma}\label{dim}
	Using the notations of Theorem \ref{pft}, we have\\
	(\rmnum{1}) $\phi_i(A) $ is $A$-invariant and $\{0\}\neq\ker \phi_i(A)\subset V_i$;\\
	(\rmnum{2}) every nonzero $A$-invariant subspace $U$ containing in some $V_i$ has dimension $r\cdot \deg(\phi_i)$ for some $r\in\{1,2,\ldots,r_i\}$ and hence has dimension at least $\deg(\phi_i)$.
\end{lemma}
\begin{proof}
	By Theorem \ref{pft}, we know that $V_i$ is nonzero. Thus, $\phi_i^{r_i}(A)$ and hence $\phi_i(A)$ is singular. This means that $\ker \phi_i(A)$ is nonzero. Clearly, $\ker \phi_i(A)\subset \ker\phi_i^{r_i}(A)=V_i$ and $\ker \phi_i(A)$ is $A$-invariant. Thus (\rmnum{1}) holds.
	
	Since $U$ is  $A$-invariant and  $U\subset V_i$, we see that  $\chi(A|_U;x)$ divides $\chi(A|_{V_i};x)$. Noting that  $\chi(A|_{V_i};x)=\phi_i^{r_i}(x)$ and  $\phi_i(x)$ is irreducible, the only factors of $\chi(A|_{V_i};x)$ are  $\phi_i^r(x)$ for $r\in \{0,1,\ldots,r_i\}$. As $U$ is a nonzero,  $\chi(A|_U;x)$  has positive degree and hence  $\chi(A|_U;x)=\phi_i^r(x)$ for positive $r\le r_i$. Thus, $\dim U=\deg \chi(A|_U;x)=r\cdot \deg \phi_i(x)$, as desired.
\end{proof}
\begin{lemma}\label{pfts}Let $A$ be a symmetric matrix over $\mathbb{F}_p$ and $U$ be a totally isotropic $A$-invariant subspace. Then, using the notations of Theorem \ref{pft},\\
(\rmnum{1}) $V_i\perp V_j$ for all distinct $i$ and $j$;\\
(\rmnum{2}) $V_i^\perp=\oplus_{j\neq i}V_j=V_1\oplus\cdots\oplus V_{i-1}\oplus V_{i+1}\cdots\oplus V_k$;\\
(\rmnum{3}) if $\phi_i(x)$ is a simple factor (i.e., $r_i=1$), then $U\cap V_i=\{0\}$;\\
(\rmnum{4}) $U=\oplus (U\cap V_i)$, where the summation is taken over all subsripts $i$ satisfying $r_i\ge 2$.
\end{lemma}
\begin{proof}
	Let $\xi$ and $\eta$ be any vectors in $V_i$ and $V_j$ respectively. By Theorem \ref{pft} (\rmnum{4}) and noting that $A^\T=A$, we have
	$$\xi^\T\eta=(h_i(A)\xi)^\T \eta=\xi^\T(h_i(A) \eta)=0.$$
	This proves (\rmnum{1}). 
	
	By (\rmnum{1}), we see that $V_i^\perp\supset\oplus_{j\neq i}V_j$. On the other hand, from the definition of orthogonal space $V_i^\perp$, we easily see that   $\dim V_i^\perp =n-\dim V_i$. Noting that $\dim(\oplus_{j\neq i}V_j)=n-\dim V_i$, the two spaces $V_i^\perp$ and $\oplus_{j\neq i}V_j$ must be equal. This proves (\rmnum{2}).
	
	Let $S=U\cap V_i$. Clearly, $S$ is $A$-invariant as both $U$ and $V_i$ are $A$-invariant. Suppose to the contrary that $S\neq \{0\}$. We first claim that $S=V_i$. Indeed, since $S$ is a subspace of $V_i$, we see that  $\chi(A|_{S},x)$ divides  $\chi(A|_{V_i},x)$.
	But  $\chi(A|_{V_i},x)$ is $\phi_i(x)$ which is irreducible, we must have  $\chi(A|_{S},x)=\chi(A|_{V_i},x)$. Thus $S=V_i$, as claimed. 	Since $V_i=S\subset U$ and $U$ is totally isotropic, we have $V_i^\perp\supset U^\perp\supset U\supset V_i$.   This together with (\rmnum{2}) implies 
	$\oplus_{j\neq i}V_j\supset V_i$, which contradicts Theorem \ref{pft} (\rmnum{1}) and hence completes the proof of (\rmnum{3}). 
	
	The last assertion clearly follows from (\rmnum{3}) and Theorem \ref{pft} (\rmnum{5}). 
\end{proof}
\subsection{Smith normal form}
An $n\times n$ matrix $U$ with integral entries is called \emph{unimodular} if $\det(U) =\pm 1$.  The following result is well known.  We restricted ourselves to square matrices for simplicity.
\begin{proposition} For every $n\times n$ integral matrix $M$ with rank $r$ over $\mathbb{Q}$, there exist unimodular matrices $U$ and $V$ such that $M = USV$,
	where $S = \diag[d_1, d_2,\ldots, d_r,0,\ldots,0]$ is a diagonal matrix with  $d_i \mid d_{i+1}$ for $i = 1, 2, \ldots, r- 1$.
\end{proposition}
The above diagonal matrix $S$ is called the \emph{Smith normal form} (SNF) of $M$, and $d_i$ with $1\le i\le r$ is called the $i$-th \emph{invariant factor} of the matrix $M$. 

For a prime $p$,  and use $\rank_p M$ and $\nullity_p M$ to denote the rank and the nullity of $M$ over $\mathbb{F}_p$, respectively. 
\begin{lemma}\label{snfbasic}
	Let $M$ be an  $n\times n$ integral matrix of rank $r$ and $S=\diag[d_1,\ldots,d_n]$ be its SNF, where $d_{r+1}=\cdots=d_n=0$  when $r<n$. Then, for any prime $p$,\\
	(\rmnum{1}) $\det M=\pm d_1d_2\cdots d_n$;\\
	(\rmnum{2}) $\rank_p M=\max\left(\{0\}\cup\{i\colon\,p\nmid d_i\}\right)$ and $\nullity_p M=\max\left(\{0\}\cup\{i\colon\,p\mid d_{n+1-i}\}\right)$;\\
	(\rmnum{3}) $Mx\equiv 0\pmod{p^2}$ has an integral solution $x\not\equiv 0\pmod{p}$ if and only if $p^2\mid d_n$.
\end{lemma}
The first two assertions of the above  lemma should be clear; for the last assertion, we refer to \cite{wang2013EJC} for a proof.
\begin{lemma}\label{Basicine}
	Let $A$ be any integral matrix and $\phi(x)$ be an irreducible factor of $\chi(A;x)$ over $\mathbb{F}_p$. Then
	\begin{equation}\label{basicine}
	\ord_p(\det \phi(A))\ge\deg \phi(x).
	\end{equation}	
\end{lemma}
\begin{proof}
Let $S=\diag[d_1,d_2,\ldots,d_n]$ be the SNF of  the integral matrix $\phi(A)$. Noting that, over $\mathbb{F}_p$, the nullity of $\phi(A)$ equals $\dim \ker\phi(A)$, Lemma \ref{dim} impliles that $\nullity_p\phi(A)\ge \deg\phi(x)$.  Let  $t= \deg\phi(x)$. Then by Lemma \ref{snfbasic} (\rmnum{1}) and (\rmnum{2}), the last $t$ diagonal entries $d_{n-t+1},\ldots,d_n$ are multiples of $p$ and hence  $\det \phi(A)$ is a multiple of $p^t$. This means $\ord_p(\det \phi(A))\ge t$, as desired.
\end{proof}
Using Lemma \ref{Basicine} for $A$ and $A+J$ simultaneously, it is not difficult to see that  if $\phi(x)$ is a common irreducible factor of $\chi(A;x)$ and $\chi(A+J;x)$ over some field $\mathbb{F}_p$,
then \begin{equation}\label{relex}
\min\{\ord_p(\det \phi(A)),\ord_p(\det \phi(A+J))\}\ge\deg \phi(x).
\end{equation}
Comparing Inequality (\ref{relex}) and Equality (\ref{keyequ}), Theorem \ref{main} states that, roughly speaking, if $G$ is not DGS, then at least one strict inequality must be witnessed in some instance of Inequality (\ref{relex}). 
\subsection{Rational regular orthogonal matrix}
An orthogonal matrix $Q$ is called \emph{regular} if $Qe=e$. We use $\rrg$ to denote the set of all rational regular orthogonal matrices of order $n$. We note that  $\rrg$ is a group under the usual matrix multiplication.  The following result is fundamental in the study of generalized spectral characterization of graphs.
\begin{lemma}[\cite{johnson1980JCTB,wang2006EUJC}]
	Let $G\in\mathcal{G}_n^c$. An $n$-vertex graph $H$ is generalized cospectral with $G$ if and only if there exists a  $Q\in \rrg$ such that 
	$Q^\T A(G) Q=A(H)$. Moreover, $Q$ is unique and satisfies $Q^\T =W(H)W^{-1} (G)$.
\end{lemma}
Define 
$$\mathcal{Q}(G)=\{Q\in \rrg\colon\, Q^\T A(G)Q \text{~is a (0,1)-matrix}\}.$$
\begin{lemma}[\cite{wang2006EUJC}] \label{dgsper}Let G be a controllable graph. Then G is DGS if and only if $\mathcal{Q}(G)$ contains only permutation matrices.
	\end{lemma}
\begin{definition}\normalfont{
	For a rational matrix $Q$, the \emph{level} of $Q$, denoted by $\ell(Q)$ or simply $\ell$, is the smallest positive integer $k$ such that $kQ$ is an integral matrix. In other words, $\ell(Q)$ is the least common multiple of denominators of all entries (as reduced fractions) in $Q$.}
\end{definition}
Clearly, a matrix $Q\in \rrg$ has level one if and only if $Q$ is a permutation matrix. Thus, we can restate Lemma  \ref{dgsper} as the following:
\begin{lemma} \label{dgslevel}Let $G\in\mathcal{G}_n^c$. Then G is DGS if and only if  $\ell(Q)=1$ for each $Q\in\mathcal{Q}(G)$.
\end{lemma}
Lemma \ref{dgslevel} explains the basic strategy to show a controllable graph $G$ to be DGS.  The key is to determine (or estimate) $\ell(Q)$ using only some properties of $G$ without  actually constructing $\mathcal{Q}(G)$.  The following three lemmas summarize some basic  results in this aspect.
\begin{lemma}[\cite{wang2006EUJC,wang2013EJC,wangyu2016,wang2017JCTB}]\label{best}
	Let $G\in\mathcal{G}_n^c$ and  $Q\in\mathcal{Q}(G)$. Letting $\ell=\ell(Q)$, then the followings hold:\\
	(\rmnum{1}) $\ell\mid d_n(W)$ and in particular  $\ell\mid \det W$, where $d_n(W)$ is the $n$-th invariant factor of $W$.\\
	(\rmnum{2}) If $p$ is a simple odd prime factor of $\det W$, then $p\nmid\ell$.\\ 
	(\rmnum{3}) Every prime factor of $\ell$ is a factor of $\Delta$.\\
	(\rmnum{4}) If  $p$ is a simple odd prime factor of $\Delta$, then $p\nmid\ell$.\\
	(\rmnum{5}) If $2^{-\lfloor\frac{n}{2}\rfloor}\det W$ is odd, then $\ell$ is odd.
	\end{lemma}
\begin{lemma}[Exclusion Condition \cite{wang2006EUJC}] 
	Let $G\in \mathcal{G}_n^c$ and $p$ be an odd prime such that $\rank_p W=n-1$. Suppose that $z_0\in \mathbb{Z}^n$ is a nontrivial (i.e., nonzero mod $p$) solution of $W^\T z\equiv 0\pmod{p}$. If $z_0^\T z_0\not\equiv 0\pmod{p}$ then $p\nmid\ell(Q)$ for any $Q\in \mathcal{Q}(G)$.
\end{lemma}
For a monic polynomial $f\in \mathbb{F}_p[x]$ with standard factorzation $f = \prod_{1\le i\le r}f_i^{e_i}$. The \emph{squarefree part} of $f$, denoted by $\sfp (f)$,  is $\prod_{1\le i\le r}f_i$; see \cite[p.~394]{gathen}.
\begin{lemma}[Improved Condition  \cite{wang2023Eujc}]
		Let $G\in \mathcal{G}_n^c$ and $p$ be an odd prime such that  $p$ is a simple factor of $d_n(W)$, the last invariant factor of $W$. If
		\begin{equation}\label{keyequ2}
		\deg\sfp(\Phi_p(G;x))= \nullity_p W,
		\end{equation}  then $p\nmid\ell(Q)$ for any $Q\in \mathcal{Q}(G)$.
\end{lemma}

\section{Proof of Theorem \ref{main}}
We prove Theorem \ref{main} by contradiction. Let $G$ be an $n$-vertex graph such that $\theta(G)$ (or equivalently $2^{-\lfloor\frac{n}{2}\rfloor}\det W$) is odd. Suppose to the contrary that $G$ is not $DGS$. Then, by Lemma \ref{dgslevel}, there exists a matrix $Q\in \mathcal{Q}(G)$ with level $\ell>1$. According to Lemma \ref{best} (\rmnum{5}), we find that $\ell$ is odd. Let $p$ be any odd prime factor of $\ell$ and we fix $p$ in the following argument. By  Lemma \ref{best} (\rmnum{1}), we know that $p$  divides $\det W$. Furthermore, it follows from Lemma \ref{best} (\rmnum{2}) that $p^2$ divides $\det W$. Similarly,  $p^2$ divides $\Delta$ by (\rmnum{3}) and (\rmnum{4}) of Lemma \ref{best}. Thus, $p^2\mid\theta(G)$.

Let $\hat{Q}=\ell\cdot Q$. We use $\col(\hat{Q})$  to denote the column space of $\hat{Q}$ over  $\mathbb{F}_p$, which is a subspace of $\mathbb{F}^n_p$.  
\begin{lemma}\label{hatQ}
	The subspace $\col(\hat{Q})\subset \mathbb{F}_p^n$ is nonzero,  totally isotropic, and  $(A+tJ)$-invariant for $t\in \{0,1\}$.
\end{lemma}
\begin{proof}
	By the minimality of $\ell$, we see that $\hat{Q}\not\equiv 0\pmod p$ and hence $\col(\hat{Q})$ is nonzero. Let $\alpha=\hat{Q}c_1\in \mathbb{Z}^n$ and $\beta=\hat{Q}c_2\in \mathbb{Z}^n$ for any $c_1,c_2\in \mathbb{Z}^n$. As $\hat{Q}^\T \hat{Q}=\ell^2 I\equiv 0\pmod{p}$, we have $\alpha^\T\beta= c_1^\T\hat{Q}^\T \hat{Q}c_2\equiv 0\pmod{p}$. Thus,  $\col(\hat{Q})$ is totally isotropic. Finally, as $Q\in \mathcal{Q}(G)$, we have $Qe=e$ and $Q^\T AQ=B$, where $B$ is a (0,1)-matrix. Clearly $A\hat{Q}=\hat{Q}B$ over $\mathbb{Z}$ and hence over $\mathbb{F}_p$. Thus $\col(\hat{Q})$ is $A$-invariant. In addition, $J\hat{Q}=\ell J\equiv 0\pmod{p}$, i.e.,  the mapping $J$, restricted to $\col(\hat{Q})$, is zero. Consequently,  $\col(\hat{Q})$ is also $A+J$-invariant.
\end{proof}
Let	$\chi(A;x)=\phi_1^{r_1}(x)\phi_2^{r_2}(x)\cdots\phi_k^{r_k}(x)$ be the standard factorization of $\chi(A;x)$ over $\mathbb{F}_p$. It follows from Lemmas \ref{pfts} and \ref{hatQ} that
	\begin{equation}
	\col(\hat{Q})=\mathop{\oplus}\limits_{i\in I}(\col(\hat{Q})\cap\ker \phi_i^{r_i}(A)),
	\end{equation}
	where $I=\{i\colon\,r_i\ge 2\}$. As $\col(\hat{Q})$ is nontrivial, we see that $I$ is nonempty and  at least one subspace $\col(\hat{Q})\cap\ker \phi_i^{r_i}(A)$ is nonzero. For such an $i$, we claim further that  $\col(\hat{Q})\cap\ker \phi_i(A)$ is also nonzero.  Let $S=\col(\hat{Q})\cap\ker \phi_i^{r_i}(A)$. Then we have the following decreasing chain of $A$-invariant subspaces:
	\begin{equation}
	\{0\}\neq S\supset \phi_i(A)S\supset\phi_i^2(A)S\supset\cdots\supset \phi_i^{r_i}(A)S=\{0\}.
	\end{equation}
	Let $j$ be the largest index in $\{0,1,\ldots,r_i\}$ such that $\phi_i^j(A)S$ is nonzero. Then $j<r_i$, and $\phi_i^{j+1}(A)S=\{0\}$ or equivalently $\ker\phi_i(A)\supset \phi_i^j(A)S$. Noting that $\col(\hat{Q})\supset S\supset \phi_i^j(A)S$, we have 
	 $S=\col(\hat{Q})\cap\ker \phi_i(A)\supset \phi_i^j(A)S$ and the claim follows. This justifies the following definition.
\begin{definition}
Let $\phi(x)$  denote a fixed multiple irreducible factor of $\chi(A;x)$ over $\mathbb{F}_p$ such that $\col(\hat{Q})\cap\ker \phi(A)$ is nonzero.
\end{definition}
\begin{lemma}\label{alsomul}
The irreducible polynomial	$\phi(x)$ is a multiple factor of $\chi(A+J;x)$ and moreover $\col(\hat{Q})\cap\ker \phi(A)=\col(\hat{Q})\cap\ker \phi(A+J)$. 
\end{lemma}
\begin{proof}
	Let $T=\col(\hat{Q})\cap\ker \phi(A)$ and $\xi$ be any nonzero vector in $T$. Noting that $T$ is $A$-invariant, we have $A^m T\subset T\subset \col(\hat{Q})$ and in particular $A^m\xi\in \col(\hat{Q})$, for any nonnegative integer $m$. As $J\hat{Q}=\ell JQ=\ell J\equiv 0\pmod{p}$ and $A^m\xi\in \col(\hat{Q})$, we see that 	$JA^m\xi=0$  over $\mathbb{F}_p$ for any $m\ge 0$. From this fact and using a simple induction on $m$ we conclude that 
	$(A+J)^m\xi=A^m\xi$ for any $m\ge 0$. Therefore, $\phi(A+J)\xi=\phi(A)\xi$ and consequently $\phi(A+J)\xi=0$ as $\xi\in \ker \phi(A)$. As $\xi\in \col(\hat{Q})$ we see that	$\xi\in \col(\hat{Q})\cap\ker \phi(A+J)$ and hence 
	\begin{equation}
		\col(\hat{Q})\cap\ker \phi(A)\subset \col(\hat{Q})\cap\ker \phi(A+J)
		\end{equation}
	by the arbitrariness of $\xi$. The reversed inclusion relation can be proved using the same argument by interchanging the role of $A$ and $A+J$. This proves the latter part of this lemma.
	
To show the former part, we first claim that $\phi(x)\mid \chi(A+J;x)$. Otherwise, $\phi(x)$ and $\chi(A+J;x)$ would be coprime (as $\phi(x)$ is irreducible) and hence there exist two polynomials $u(x),v(x)\in \mathbb{F}_p[x]$ such that 
	$u(x)\phi(x)+v(x)\chi(A+J;x)=1$, which implies
	 $$u(A+J)\phi(A+J)\xi+v(A+J)\chi(A+J;A+J)\xi=\xi.$$
	 The first term in the sum is 0 as we have shown  $\phi(A+J)\xi=0$. By Caylay-Hamilton Theorem, we know $\chi(A+J;A+J)=0$ and hence the second term in the sum is also zero. Thus, $\xi=0$, contradicting our assmption on $\xi$. This shows the claim. 
	 
	 It remains to show that the factor $\phi(x)$ of $\chi(A+J;x)$  is not simple. Note that $\col(\hat{Q})$ is totally isotropic and $(A+J)$-invariant by Lemma \ref{hatQ}.  If  $\phi(x)$ is simple, then it follows from Lemma \ref{pfts} (\rmnum{3}) that 
	 \begin{equation}
	 \col(\hat{Q})\cap\ker \phi(A+J)=\{0\}.
	 \end{equation}
But  $\col(\hat{Q})\cap\ker \phi(A+J)$ equals $\col(\hat{Q})\cap\ker \phi(A)$, which is nonzero. This contradiction completes the proof.	 
\end{proof}
Recall that $\Phi_p(G;x)=\gcd(\chi(A;x),\chi(A+J),x)$ over $\mathbb{F}_p$. As $\phi(x)$ is multiple factor of both $ \chi(A;x)$ and $ \chi(A+J;x)$, we easily see that  $\phi(x)$ is a multiple factor of $\Phi_p(G;x)$. We record it as the following.
\begin{proposition}
	The polynomial $\phi(x)$ is a multiple irreducible factor of $\Phi_p(G;x)$.
\end{proposition}
\begin{proposition}\label{strict} The following strict inequality hold:
\begin{equation}
\min\{\ord_p(\det \phi(A)),\ord_p(\det \phi(A+J))\}>\deg \phi(x).
\end{equation}
\end{proposition}
\begin{proof}
	We only prove $\ord_p(\det \phi(A))>\deg \phi(x)$ and the corresponding result for $A+J$ can be proved in the same way. Let $T=\col(\hat{Q})\cap\ker \phi(A)$.
	Since $T$ is  nonzero  and $A$-invariant, the dimension of $T$ is at least $\deg \phi(x)$ by Lemma \ref{dim}.  If  $\dim \ker \phi(A)>\deg \phi(x)$ then 
a similar argumment as in the proof of Lemma \ref{Basicine} shows that $\ord_p(\det \phi(A))\ge \dim \ker \phi(A)>\deg \phi(x)$. Thus we may assume $\dim \ker \phi(A)=\deg \phi(x)$. Consequently, $T=\ker \phi(A)$, i.e., $\ker \phi(A)\subset \col(\hat{Q})$.

Let $\xi,\eta$ be any vector in $\mathbb{Z}^n\setminus\{0\}$  such that $ \phi(A)\xi\equiv\phi(A)\eta\equiv 0\pmod{p}$. We claim 
\begin{equation}
 \xi^\T\phi(A)\eta\equiv 0\pmod{p^2}.
\end{equation}
As $\ker \phi(A)\subset \col(\hat{Q})$, we see that  $\xi,\eta\in \col(\hat{Q})$  over $\mathbb{F}_p$ and hence we can write
\begin{equation} 
\xi=\hat{Q}u+p\alpha,  \eta=\hat{Q}v+p\beta,  
\end{equation}
where $u,v,\alpha,\beta\in \mathbb{Z}^n$. Consequently, noting that $A$ is symmetric, we have
\begin{eqnarray}\label{modp2}
 \xi^\T\phi(A)\eta &=&(\hat{Q}u+p\alpha)^\T\phi(A)(\hat{Q}v+p\beta)\nonumber\\
&\equiv&u^\T(\hat{Q}^\T \phi(A)\hat{Q})v+p\alpha^\T\phi(A)\hat{Q}v+p\beta^\T\phi(A)\hat{Q}u\pmod{p^2}.
\end{eqnarray}
As  $p\mid \ell$ and $Q^\T A^m Q$ is integral for any $m\ge 0$, we easily see that $\hat{Q}^\T \phi(A)\hat{Q}=\ell^2 Q\phi(A)\hat{Q}\equiv 0\pmod{p^2}$. As $\ker \phi(A)\subset \col(\hat{Q})$, we obtain $\phi(A)(\hat{Q}v)\equiv 0\pmod{p}$ and hence $p\alpha^\T\phi(A)\hat{Q}v\equiv 0\pmod{p^2}$. Similarly, we also have
$p\beta^\T\phi(A)\hat{Q}u\equiv 0\pmod{p^2}$. It follows from Eq. \eqref{modp2} that  $\xi^\T\phi(A)\eta\equiv 0\pmod{p^2}$, as claimed.

Write $s=\dim \ker \phi(A)$ and let $\xi_1,\ldots,\xi_s$ be $s$ integral vectors such that $\xi_1,\ldots,\xi_s$ constitute a base of $\ker \phi(A)$ over $\mathbb{F}_p$. 
Then by the above claim, $\xi_i^\T\phi(A)\eta\equiv 0\pmod{p^2}$, i.e., $\xi_i^\T\frac{\phi(A)\eta}{p}\equiv 0\pmod{p}$ for $i=1,\ldots,s$. Noting that $\xi_i^\T\phi(A)=(\phi(A)\xi_i)^\T\equiv 0\pmod{p}$, we obtain
\begin{equation}
\begin{pmatrix}
\xi_1^\T\\
\xi_2^\T\\
\vdots\\
\xi_s^\T
\end{pmatrix} \left[\phi(A), \frac{\phi(A)\eta}{p}\right]\equiv 0\pmod{p}.
\end{equation}
Thus $\rank_p[\phi(A), \frac{\phi(A)\eta}{p} ]\le n-s$. Noting that  $\rank_p \phi(A)=n-s$, we  have $\rank_p[\phi(A), \frac{\phi(A)\eta}{p} ]=\rank_p \phi(A)$ and hence there exists an integral vector $\gamma$ such that
\begin{equation}
 \frac{\phi(A)\eta}{p}\equiv\phi(A)\gamma\pmod{p}.
\end{equation}
Multiplying $p$ on both sides generates  
\begin{equation}
\phi(A)(\eta-p\gamma)\equiv 0\pmod{p^2}.
\end{equation}
We may assume $\phi(A)$ has full rank over $\mathbb{Q}$ since otherwise $\ord_p\phi(A)=\infty$ and we are done. Let $d_i$ be the $i$-th invariant factor of $\phi(A)$ for $i=1,\ldots,n$. Since $\eta\not\equiv 0\pmod{p}$, we find that $p^2\mid d_n$ by Lemma \ref{snfbasic}. Moreover, as $\dim\ker \phi(A)=\deg\phi(x)=s$, the last $s$ invariant factors are multiple of $p$. Since $\det\phi(A)=\pm d_1\cdots d_n$, we find that $p^{s+1}\mid \det \phi(A)$, i.e., $\ord_p(\det\phi(A))>\deg\phi(x)$. This completes the proof.
\end{proof}
\begin{proof}[Proof of Theorem \ref{main}]
Proposition \ref{strict} contradicts Eq.~\eqref{keyequ} and hence completes the proof of Theorem \ref{main}.
\end{proof}
\section{Examples}
We first give an example for which the DGS-property can  be guaranteed by Theorem \ref{main}  while the previous criteria including Theorem \ref{old}, Exclusion Condition and Improved Condition all fail. We use Mathematica for the computation.

\noindent\textbf{Example 1} Let $n=12$ and $G$ be the graph with adjacency matrix
\begin{equation}
A=\begin{tiny}\left(
	\begin{array}{cccccccccccc}
	0 & 0 & 0 & 0 & 1 & 1 & 1 & 0 & 1 & 0 & 1 & 0 \\
	0 & 0 & 0 & 0 & 0 & 1 & 1 & 1 & 1 & 0 & 0 & 1 \\
	0 & 0 & 0 & 1 & 1 & 1 & 1 & 1 & 1 & 1 & 1 & 0 \\
	0 & 0 & 1 & 0 & 1 & 0 & 0 & 0 & 1 & 0 & 1 & 1 \\
	1 & 0 & 1 & 1 & 0 & 1 & 0 & 0 & 1 & 0 & 0 & 0 \\
	1 & 1 & 1 & 0 & 1 & 0 & 0 & 1 & 0 & 1 & 0 & 1 \\
	1 & 1 & 1 & 0 & 0 & 0 & 0 & 1 & 1 & 1 & 0 & 1 \\
	0 & 1 & 1 & 0 & 0 & 1 & 1 & 0 & 0 & 0 & 1 & 0 \\
	1 & 1 & 1 & 1 & 1 & 0 & 1 & 0 & 0 & 1 & 0 & 1 \\
	0 & 0 & 1 & 0 & 0 & 1 & 1 & 0 & 1 & 0 & 1 & 1 \\
	1 & 0 & 1 & 1 & 0 & 0 & 0 & 1 & 0 & 1 & 0 & 0 \\
	0 & 1 & 0 & 1 & 0 & 1 & 1 & 0 & 1 & 1 & 0 & 0 \\
	\end{array}
	\right).\end{tiny}
\end{equation}
It can be computed that
\begin{eqnarray*}
	\theta(G)&=&\gcd\{2^{-\lfloor\frac{n}{2}\rfloor}\det W,\Delta(G)\}\\
           	&=&\gcd\{-20514573,424319456090918385320095315960579067904\}\\
           	&=&3^3,
\end{eqnarray*}
which is odd but not squarefree. Thus, Theorem \ref{old} does not apply to this example.  In addition, as the SNF of $W(G)$ is 
\begin{equation}
\diag[\underbrace{1,1,1,1,1,1}_6,\underbrace{2,2,2,2,2\times 3,2\times 3^2\times759799}_6].
\end{equation}
we find that neither the Exclusion Condition nor the Improved Condition can be used to eliminate the possibility of $3\mid \ell(Q)$ for some $Q\in \mathcal{Q}(G)$. We turn to Theorem \ref{main}. Note that 
$\Phi_3(G;x)$ factors as $(x+1)^2$ and hence the only multiple irreducible factor of $\Phi_3(G;x)$ is $\phi(x)=x+1$. As
$\det (\phi(A))=\det(A+I)=-3\times 5^2$ and $\det (\phi(A+J))=\det(A+J+I)=-3^2\times 23$, we find that Eq. \eqref{keyequ} holds. Thus $G$ is DGS by Theorem \ref{main}.

Example 1 illustrates the main advantage of Theorem \ref{main}, compared to the Exclusion Condition and the Improved Condition.  We do not need the assumption that $\rank_p(W)=n-1$ or $p$ is a \emph{simple} factor $d_n(W)$. For the case that $p$ is a simple factor of $d_n(W)$, Theorem \ref{main} seems more powerful than the Improved Condition. We give a small example for illustration.

\noindent\textbf{Example 2} Let $n=10$ and $G$ be the graph with adjacency matrix
\begin{equation}
A=\begin{tiny}\left(
\begin{array}{cccccccccc}
0 & 0 & 1 & 0 & 1 & 0 & 1 & 0 & 1 & 0 \\
0 & 0 & 1 & 1 & 1 & 1 & 1 & 1 & 1 & 0 \\
1 & 1 & 0 & 0 & 0 & 0 & 1 & 0 & 1 & 1 \\
0 & 1 & 0 & 0 & 1 & 1 & 0 & 1 & 1 & 1 \\
1 & 1 & 0 & 1 & 0 & 0 & 0 & 0 & 0 & 0 \\
0 & 1 & 0 & 1 & 0 & 0 & 0 & 1 & 1 & 0 \\
1 & 1 & 1 & 0 & 0 & 0 & 0 & 1 & 1 & 0 \\
0 & 1 & 0 & 1 & 0 & 1 & 1 & 0 & 1 & 0 \\
1 & 1 & 1 & 1 & 0 & 1 & 1 & 1 & 0 & 0 \\
0 & 0 & 1 & 1 & 0 & 0 & 0 & 0 & 0 & 0 \\
\end{array}
\right).\end{tiny}
\end{equation}
The Smith normal form of $W$ is
\begin{equation}
\diag[\underbrace{1,1,1,1,1}_5,\underbrace{2, 2, 2\times 3, 2\times 3, 2\times 3\times 7\times 19}_5 ].
\end{equation}
As the last invariant is squarefree, we  try to use the Improved Condition to eliminate $p=3$ as a possible factor of $\ell(Q)$ for $Q\in \mathcal{Q}(G)$. 
It can be computed that $\Phi_3(G;x)=(x+1)^3$ and $\sfp(\Phi_3(G;x))=(x+1)$. Thus, $\deg \sfp(\Phi_3(G;x))=1<3=\nullity_p W$ and the Improved Condition fails.

As $\theta(G)=3^3$ and the only multiple irreducible factor of $\Phi_3(G;x)$ is $x+1$, we need to check Eq. \eqref{keyequ} for $p=3$ and $\phi(x)=x+1$. Indeed,
$\det(A+I)=6$, $\det(A+J+I)=12$ and hence
\begin{equation}
\min\{\ord_p(\det\phi(A)),\ord_p(\det\phi(A+J))\}=1=\deg \phi(x).
\end{equation}
Thus, $G$ is DGS by Theorem \ref{main}.

The following example shows that Theorem \ref{main} is tight in some sense, i.e., if the condition in Eq. \eqref{keyequ} fails, then $G$ may not be DGS.\\

\noindent\textbf{Example 3} Let $n=14$ and $G$ be the graph with adjacency matrix
\begin{equation}
A=\begin{tiny}\left(
\begin{array}{cccccccccccccc}
0 & 1 & 1 & 0 & 1 & 0 & 0 & 1 & 0 & 0 & 1 & 1 & 0 & 0 \\
1 & 0 & 1 & 1 & 1 & 0 & 0 & 0 & 1 & 0 & 0 & 0 & 0 & 1 \\
1 & 1 & 0 & 0 & 1 & 0 & 1 & 1 & 1 & 0 & 0 & 0 & 0 & 1 \\
0 & 1 & 0 & 0 & 1 & 1 & 0 & 1 & 0 & 1 & 1 & 0 & 1 & 1 \\
1 & 1 & 1 & 1 & 0 & 1 & 1 & 0 & 0 & 1 & 0 & 1 & 1 & 0 \\
0 & 0 & 0 & 1 & 1 & 0 & 1 & 1 & 0 & 0 & 1 & 0 & 1 & 1 \\
0 & 0 & 1 & 0 & 1 & 1 & 0 & 1 & 1 & 0 & 1 & 1 & 0 & 0 \\
1 & 0 & 1 & 1 & 0 & 1 & 1 & 0 & 1 & 1 & 0 & 0 & 0 & 1 \\
0 & 1 & 1 & 0 & 0 & 0 & 1 & 1 & 0 & 1 & 0 & 1 & 1 & 1 \\
0 & 0 & 0 & 1 & 1 & 0 & 0 & 1 & 1 & 0 & 1 & 1 & 0 & 1 \\
1 & 0 & 0 & 1 & 0 & 1 & 1 & 0 & 0 & 1 & 0 & 1 & 1 & 1 \\
1 & 0 & 0 & 0 & 1 & 0 & 1 & 0 & 1 & 1 & 1 & 0 & 1 & 1 \\
0 & 0 & 0 & 1 & 1 & 1 & 0 & 0 & 1 & 0 & 1 & 1 & 0 & 0 \\
0 & 1 & 1 & 1 & 0 & 1 & 0 & 1 & 1 & 1 & 1 & 1 & 0 & 0 \\
\end{array}
\right).\end{tiny}
\end{equation}
Direct computation indicates that $\theta(G)=3^8\times 5^2$ and the Smith normal form of $W$ is
\begin{equation}
\diag[\underbrace{1,\ldots, 1}_6,\underbrace{ 2, 2, 2, 2\times 3, 2\times 3, 2\times 3\times 5, 2\times 3^5\times 5^2\times 7\times 31\times 461\times 787}_6].
\end{equation}
We need to check Eq.~\eqref{keyequ} for each multiple prime factor of $\theta(G)$ and for each multiple irreducible factor of $\Phi_p(G)$. We summarize the procedure in Table \ref{check}.
\begin{table}[htbp]
	\footnotesize
	\centering
	\caption{\label{check}  Using Theorem \ref{main} for Example 3}
	\begin{threeparttable}
		\begin{tabular}{ccccc}
			\toprule
			$p$ &$\Phi_p(G;x)$&multiple irreducible factor of $\Phi_p(G;x)$ &whether Eq.~\eqref{keyequ} holds\\
			\midrule
			3  &$(x+1)^2 (x+2)^3$  &$x+1$ & F\\
			3  & $(x+1)^2 (x+2)^3$ &$x+2$ & T\\
			5  & $(x+2)^2$& $x+2$ &T\\
			\bottomrule
		\end{tabular}
	\end{threeparttable}
\end{table}
From Table \ref{check}, we can eliminate the possibility of $p\mid\ell(Q)$ for $p=5$ but not for $p=3$. Actually, let $Q$ be a rational orthogonal matrix as shown below.
Then $Q \in\mathcal{Q}(G)$  with level $\ell(G) = 3$, since it can be easily verified that $Q^\T AQ$ is a (0, 1)-matrix. Hence, $G$ is not DGS.
\begin{equation}
Q=\frac{1}{3}
\begin{tiny}\left(
\begin{array}{cccccccccccccc}
0 & 0 & 0 & 0 & 0 & 0 & 0 & 0 & 1 & 1 & 1 & 2 & -1 & -1 \\
3 & 0 & 0 & 0 & 0 & 0 & 0 & 0 & 0 & 0 & 0 & 0 & 0 & 0 \\
0 & 0 & 0 & 0 & 0 & 0 & 0 & 0 & 2 & -1 & -1 & 1 & 1 & 1 \\
0 & 3 & 0 & 0 & 0 & 0 & 0 & 0 & 0 & 0 & 0 & 0 & 0 & 0 \\
0 & 0 & 3 & 0 & 0 & 0 & 0 & 0 & 0 & 0 & 0 & 0 & 0 & 0 \\
0 & 0 & 0 & 3 & 0 & 0 & 0 & 0 & 0 & 0 & 0 & 0 & 0 & 0 \\
0 & 0 & 0 & 0 & 3 & 0 & 0 & 0 & 0 & 0 & 0 & 0 & 0 & 0 \\
0 & 0 & 0 & 0 & 0 & 3 & 0 & 0 & 0 & 0 & 0 & 0 & 0 & 0 \\
0 & 0 & 0 & 0 & 0 & 0 & 0 & 0 & 1 & 1 & 1 & -1 & 2 & -1 \\
0 & 0 & 0 & 0 & 0 & 0 & 3 & 0 & 0 & 0 & 0 & 0 & 0 & 0 \\
0 & 0 & 0 & 0 & 0 & 0 & 0 & 0 & 1 & 1 & 1 & -1 & -1 & 2 \\
0 & 0 & 0 & 0 & 0 & 0 & 0 & 3 & 0 & 0 & 0 & 0 & 0 & 0 \\
0 & 0 & 0 & 0 & 0 & 0 & 0 & 0 & -1 & 2 & -1 & 1 & 1 & 1 \\
0 & 0 & 0 & 0 & 0 & 0 & 0 & 0 & -1 & -1 & 2 & 1 & 1 & 1 \\
\end{array}
\right).\end{tiny}
\end{equation}
 
Table \ref{comp} gives some experimental results to see the effectiveness of Theorem \ref{main},  compared to Theorem \ref{old} (together with the Exclusion Condition and Improved Condition). Using Mathematica, for each $n \in \{10, 15, . . . , 50\}$, we randomly generate 10,000 graphs using the random graph model $G(n,p)$ with $p=1/2$. For example, for $n=50$, among 10,000 graphs generated in one experiment, 3409 graphs have  odd $\theta(G)$. For these graphs, 2780 graphs satisfy the condition of Theorem 1 and this number can be further increased to  2996 if the Exclusion Condition and/or the Improved Condition has been employed. The corresponding number is 3118 by using the newly obtained criterion  (Theorem \ref{main}).

\begin{table}[htbp]
	\footnotesize
	\centering
	\caption{\label{comp} Comparison between Theorem \ref{old} and Theorem \ref{main}}
	\begin{threeparttable}
		\begin{tabular}{ccccc}
			\toprule
			$n$ &\# graphs & \#DGS & \#DGS  &\#DGS\\
			(graph order) & (with $\theta(G)$ odd) &(by Theorem \ref{old})  & (by Theorem \ref{old}+EC+IC)\tnote{*} &(by Theorem \ref{main})\\
			\midrule
			10  &  3300 & 2964 &3082 &3124\\
			15  &  3402 & 2768 &2978 &3095\\
			20  &  3413 &2771 &2998&3131\\
			25  &  3415 &2785 &2986&3122\\
			30  &  3508 & 2851&3084&3225\\
			35  & 3457  &2803 &3023&3161\\
			40  & 3371 & 2721&2954&3086\\
			45  & 3383  &2774 &2979&3122\\
			50  & 3409  & 2780&2996&3118\\
			\bottomrule
		\end{tabular}
	\begin{tablenotes}
		\footnotesize
		\item[*]EC=Excusion Condition; IC=Improved Condition.
	\end{tablenotes}	
	\end{threeparttable}
\end{table}
	\section*{Acknowledgments}
This work is partially supported by the National Natural Science Foundation of China (Grant Nos. 12001006, 11971406 and 11971376).

\end{document}